\documentclass[12pt,notitlepage]{amsart}
\usepackage{latexsym,amsfonts,amssymb,amsmath,amsthm}
\usepackage{color}
\newcommand{\mr}{\ensuremath{\mathbb R}}

\newcommand{\half}{\ensuremath{ \frac{1}{2}}}

\newcommand{\intR}{\int_{-\infty}^{\infty}}

\newcommand{\thalf}{\tfrac12}

\newcommand{\leg}[2]{\left(\frac{#1}{#2}\right)}

\newtheorem*{thm1.1'}{Theorem 1.1'}

\theoremstyle{plain}		
	\newtheorem{mytheo}{Theorem}

     \newtheorem{mylemma}[mytheo]{Lemma}

\theoremstyle{remark}

\numberwithin{equation}{section}
\begin{document}
\title{A short proof of Levinson's theorem}
\author{Matthew P. Young} 
\address{Department of Mathematics \\
 	  Texas A\&M University \\
 	  College Station \\
	  TX 77843-3368 \\
		U.S.A.}
\curraddr{School of Mathematics \\
Institute for Advanced Study \\
Einstein Drive \\ Princeton, NJ 08540 USA}
\thanks{This material is based upon work supported by the National Science Foundation under agreement Nos. DMS-0758235 and DMS-0635607.  Any opinions, findings and conclusions or recommendations expressed in this material are those of the authors and do not necessarily reflect the views of the National Science Foundation.}
\email{myoung@math.tamu.edu}

\maketitle
\section{Introduction}
In 1974, Levinson \cite{Levinson} proved that $1/3$ of the zeros of the Riemann zeta function $\zeta(s)$ lie on the critical line.  Apparently his work has a reputation for being difficult, and many textbook authors (\cite{T}, \cite{Ivic}, \cite{KV}, \cite{IK}) present Selberg's method \cite{Selberg} instead (which gives a very small positive percent of zeros).  Here we show how innovations in the subject can greatly simplify the proof of Levinson's theorem.

To set some terminology, let $N(T)$ denote the number of zeros $\rho = \beta + i\gamma$ with $0 < \gamma < T$, and let $N_0(T)$ denote the number of such critical zeros with $\beta = 1/2$.  Define $\kappa$ by
 $\kappa = \liminf_{T \rightarrow \infty} \frac{N_0(T)}{N(T)}$. 
Levinson's result is that $N_0(T) > \frac13 N(T)$ for $T$ sufficiently large.

The basic technology to prove that many zeros lie on the critical line is an asymptotic for a mollified 
second moment of the zeta function (and its derivative).  This is well-known, and clear presentations can be found in various sources (\cite{Levinson}, \cite{Conrey25}, etc.).  We briefly summarize the setup.  Let $L = \log{T}$, and suppose $Q(x)$ is a real polynomial satisfying $Q(0) = 1$. Set 
\begin{equation*}
 V(s) = Q\Big(-\frac{1}{L} \frac{d}{ds} \Big) \zeta(s).
\end{equation*}
Levinson's original approach naturally had $Q(x) = 1-x$, but Conrey \cite{ConreyJNT} showed how more general choices of $Q$ can be used to improve results.  For historical comparison we shall eventually choose $Q(x) = 1-x$.
Let $\sigma_0 = \half - R/L$ for $R$ a positive real number to be chosen later, $M= T^{\theta}$ for some $0 < \theta < \half$, and $P(x) = \sum_j a_j x^j$ be a real polynomial satisfying $P(0) = 0$, $P(1) = 1$. 
Suppose that $\psi$ is a mollifier of the form
\begin{equation*}
 \psi(s) = \sum_{h \leq M} \frac{\mu(h)  
}{h^{s+\frac12 - \sigma_0}} P\Big(\frac{\log{M/h}}{\log{M}}\Big),
\end{equation*}
 Again, for historical reasons we eventually take $P(x) = x$. 
The conclusion is that
\begin{equation}
\label{eq:kappa}
 \kappa \geq 1 - \frac{1}{R} \log \Big( \frac{1}{T} \int_{1}^{T} |V \psi(\sigma_0 + it)|^2 dt \Big) + o(1).
\end{equation}

The evaluation of the mollified second moment of zeta appearing in \eqref{eq:kappa} is considered to be the difficult part of Levinson's proof (taking up over $30$ pages in \cite{Levinson}).
Conrey and Ghosh \cite{CG} gave a simpler proof.
Here we show how to obtain this asymptotic in an easier way.

\begin{mytheo}
\label{thm:moment}
We have
\begin{equation}
\label{eq:moment}
 \frac{1}{T} \int_1^{T} |V\psi(\sigma_0 + it)|^2 dt = c(P,Q,R, \theta) + o(1),
\end{equation}
as $T \rightarrow \infty$, where
\begin{equation}
\label{eq:cdef}
 c(P, Q, R, \theta) = 1 + \frac{1}{\theta} \int_0^{1} \int_0^{1} e^{2Rv} \Big( \frac{d}{dx} \left. e^{R\theta x} 
P(x+u) Q(v + \theta x ) \right|_{x=0} \Big)^2
 du dv.
\end{equation}
\end{mytheo}
\noindent With $P(x) = x$, $Q(x) = 1-x$, $R=1.3$, $\theta = .5$, and using any standard computer package,
\begin{equation*}
 c(P, Q, R, \theta) = 2.35\dots, \quad \text{and }  \kappa \geq 0.34\dots.
\end{equation*}



\section{A smoothing argument}
\label{section:smoothing}
To simplify forthcoming arguments, it is preferable to smooth the integral in \eqref{eq:moment}.  Suppose that $w(t)$ is a smooth function satisfying the following properties:
\begin{align}
\label{eq:w1}
&0 \leq w(t) \leq 1 \text{ for all } t \in \mr, \\
&w \text{ has compact support in } [T/4, 2T], \\
&w^{(j)}(t) \ll_j \Delta^{-j}, \text{ for each } j=0,1,2, \dots, \quad \text{where } \Delta = T/L.
\label{eq:w3}
\end{align}
\begin{mytheo}
\label{thm:smoothed}
 For any $w$ satisfying \eqref{eq:w1}-\eqref{eq:w3}, and $\sigma = 1/2 -R/L$,
\begin{equation}
\label{eq:smoothedintegral}
 \intR w(t) |V\psi(\sigma +it)|^2 dt =  c(P, Q, R, \theta) \widehat{w}(0) + O(T/L),
\end{equation}
uniformly for $R \ll 1$, where $c(P, Q, R, \theta)$ is given by \eqref{eq:cdef}.
\end{mytheo}
We briefly explain how to deduce Theorem \ref{thm:moment} from Theorem \ref{thm:smoothed}.  
By choosing $w$ to satisfy \eqref{eq:w1}-\eqref{eq:w3} and in addition to be an upper bound for the characteristic function of the interval $[T/2, T]$, and with support in $[T/2 - \Delta, T + \Delta]$, we get
\begin{equation*}
 \int_{T/2}^{T} |V \psi(\sigma_0 + it)|^2 dt \leq c(P, Q, R, \theta) \widehat{w}(0) + O(T/L).
\end{equation*}
Note $\widehat{w}(0) = T/2 + O(T/L)$.
We similarly get a lower bound.  Summing over dyadic segments gives the full integral.

\section{The mean-value results}
Rather than working directly with $V(s)$, instead consider the following general integral:
\begin{equation}
\label{eq:Idef}
 I_{}(\alpha,\beta) = \intR w(t) \zeta(\tfrac12 + \alpha + it) \zeta(\tfrac12+ \beta -it) |\psi (\sigma_0 + it)|^2 dt,
\end{equation}
where $\alpha, \beta \ll L^{-1}$ (with any fixed implied constant).
The main result is
\begin{mylemma}
\label{lemma:Ialphabeta}
We have
\begin{equation}
I(\alpha, \beta) = c(\alpha,\beta) \widehat{w}(0) + O(T/L),
\end{equation}
uniformly for $\alpha, \beta \ll L^{-1}$, where
\begin{equation}
\label{eq:c}
 c(\alpha, \beta)=1 + \frac{1}{\theta} \frac{d^2}{dx dy} 
 M^{-\beta x - \alpha y} \int_0^{1} \int_0^{1} T^{-v(\alpha+\beta)} P(x+u) P(y+u) du \Big|_{x=y=0}.
\end{equation}
\end{mylemma} 
\begin{proof}[Proof that Lemma \ref{lemma:Ialphabeta} implies Theorem \ref{thm:smoothed}]
Define $I_{\text{smooth}}$ to be the left hand side of \eqref{eq:smoothedintegral}.  Then
\begin{equation}
\label{eq:Ismoothderiv}
I_{\text{smooth}} =   Q\Big(-\frac{1}{L} \frac{d}{d\alpha}\Big) Q\Big(-\frac{1}{L} \frac{d}{d\beta}\Big) I(\alpha, \beta) \Big|_{\alpha=\beta=-R/L}.
\end{equation}
We first argue that we can obtain $c(P, Q, R, \theta)$ by applying the above differential operator to $c(\alpha, \beta)$.  
Since $I(\alpha, \beta)$ and $c(\alpha, \beta)$ are holomorphic with respect to $\alpha$, $\beta$ small, the derivatives appearing in \eqref{eq:Ismoothderiv} can be obtained as integrals of radii $\asymp L^{-1}$ around the points $-R/L$, from Cauchy's integral formula.  Since the error terms hold uniformly on these contours, the same error terms that hold for $I(\alpha, \beta)$ also hold for 
$I_{\text{smooth}}$.

Next we check that applying the differential operator to $c(\alpha, \beta)$ does indeed give \eqref{eq:cdef}.
Notice the simple formula
\begin{equation}
\label{eq:Qop}
 Q\left(\frac{-1}{\log{T}} \frac{d}{d\alpha}\right) X^{-\alpha} = Q\leg{\log{X}}{\log{T}} X^{-\alpha}.
\end{equation}
Using \eqref{eq:Qop} we have
\begin{multline*}
Q(-\frac{1}{L} \frac{d}{d\alpha}) Q(-\frac{1}{L} \frac{d}{d\beta}) c(\alpha, \beta)
\\
= 1 + \frac{1}{\theta} \frac{d^2}{dx dy}  M^{-\beta x -\alpha y} \int_0^{1} \int_0^{1} T^{-v(\alpha + \beta)} P(x+u) P(y+u) Q(v + x\theta) Q(v + y \theta) du dv \Big|_{x=y=0},
\end{multline*}
which after evaluating at $\alpha = \beta = -R/L$ and simplifying becomes
\begin{equation*}
1 + \frac{1}{\theta} \frac{d^2}{dx dy}  e^{R\theta(x+y)} 
\int_0^{1} \int_0^{1} e^{2Rv} P(x+u) P(y+u) Q(v + \theta x) Q(v + \theta y)
 du dv \Big|_{x=y=0} .
\end{equation*}
This simplifies to give the right hand side of \eqref{eq:cdef}, as desired.
\end{proof}

\section{Two lemmas}
A variation on the standard approximate functional equation (\cite{IK} Theorem 5.3) gives
\begin{mylemma}
\label{lemma:AFE}
Let $G(s) = e^{s^2} p(s)$ where $p(s) = \frac{(\alpha+\beta)^2-(2s)^2}{(\alpha + \beta)^2}$,
and define
\begin{equation}
\label{eq:V}
V_{\alpha, \beta}(x, t) = \frac{1}{2 \pi i} \int_{(1)} \frac{G(s)}{s} g_{\alpha, \beta}(s,t) x^{-s} ds, \quad
g_{\alpha, \beta}(s,t) = \pi^{-s}
\frac{\Gamma\left(\frac{\half + \alpha + s +it}{2} \right)}{\Gamma\left(\frac{\half + \alpha +it }{2} \right)} 
\frac{\Gamma\left(\frac{\half + \beta + s -it}{2} \right)}{\Gamma\left(\frac{\half + \beta -it }{2} \right)} .
\end{equation}
Furthermore, set
\begin{equation*}
X_{\alpha,\beta,t} = \pi^{\alpha + \beta} 
\frac{\Gamma(\frac{\half -\alpha - it}{2})}{\Gamma(\frac{\half + \alpha + it}{2})}
\frac{\Gamma(\frac{\half -\beta + it}{2})}{\Gamma(\frac{\half + \beta - it}{2})}
.
\end{equation*}
Then if $\alpha, \beta$ have real part less than $1/2$, and for any $A \geq 0$, we have
\begin{multline*}
\zeta({\textstyle \half + \alpha + it}) 
\zeta({\textstyle \half + \beta - it}) 
=
\sum_{m,n} \frac{1}{m^{\half + \alpha} n^{\half + \beta}} \left(\frac{m}{n}\right)^{-it} V_{\alpha, \beta} ( mn, t)
\\
+ X_{\alpha,\beta,t}
   \sum_{m,n} \frac{1}{m^{\half -\beta} n^{\half - \alpha}} \left(\frac{m}{n}\right)^{-it} V_{-\beta, -\alpha} (mn, t) + O_A((1 + |t|)^{-A}).
\end{multline*}
\end{mylemma}
\noindent {\bf Remark}.  Stirling's approximation gives for $t$ large and $s$ in any fixed vertical strip
\begin{equation}
\label{eq:Stirling}
 X_{\alpha,\beta,t} = \leg{t}{2 \pi }^{-\alpha-\beta}(1 + O(t^{-1})), \qquad g_{\alpha,\beta}(s,t) = \leg{t}{2\pi}^s(1 +O(t^{-1} (1+|s|^2))).
\end{equation}
Furthermore, for any $A \geq 0$ and $j=0,1,2, \dots$, we have uniformly in $x$,
\begin{equation}
\label{eq:Vbound}
t^j \frac{\partial^j}{\partial t^j} V_{\alpha, \beta}(x,t) \ll_{A,j} (1+ |t/x|)^{-A}.
\end{equation}

\begin{mylemma}
\label{lemma:twisted}
 Suppose $w$ satisfies \eqref{eq:w1}-\eqref{eq:w3}, and that $h,k$ are positive integers with $hk \leq T^{2\theta}$ with $\theta < 1/2$, and $\alpha, \beta \ll L^{-1}$.  Then
\begin{multline}
\label{eq:twisted}
 \intR w(t) \big(\tfrac{h}{k}\big)^{-it}  \zeta(\thalf + \alpha + it) \zeta(\thalf + \beta - it) dt 
= 
\sum_{hm = k n} \frac{1}{m^{\half + \alpha} n^{\half + \beta}} \intR V_{\alpha, \beta}(mn, t) w(t) dt
\\
+ \sum_{h m = k n} \frac{1}{m^{\half - \beta} n^{\half - \alpha}} \intR V_{-\beta, -\alpha}(mn, t) X_{\alpha, \beta, t} w(t) dt + O_{A, \theta} (T^{-A}).
\end{multline}
\end{mylemma}
\begin{proof}
We apply Lemma \ref{lemma:AFE} to the left hand side.  It suffices by symmetry to consider the first part of the approximate functional equation, giving
\begin{equation*}
\sum_{m,n} \frac{1}{m^{\half + \alpha} n^{\half + \beta}} \intR w(t) \left(\frac{hm}{kn}\right)^{-it} V_{\alpha, \beta} ( mn , t) dt.
\end{equation*}
The terms with $hm=kn$ visibly give the first term on the right hand side of \eqref{eq:twisted}.  
By combining \eqref{eq:w3} with \eqref{eq:Vbound}, note that we have uniformly in $x$ that
\begin{equation*}
\frac{\partial^j}{\partial t^j} w(t) V_{\alpha, \beta}(x, t) \ll_{j,A} (1+ |x/T|)^{-A} \Delta^{-j}.
\end{equation*}
Hence for $hm \neq kn$, we have by repeated integration by parts that
\begin{equation*}
\intR w(t) \leg{hm}{kn}^{-it} V_{\alpha, \beta}(mn, t)  dt \ll_{j,A} \frac{(1+ \frac{mn}{T})^{-A}}{\Delta^j |\log{\frac{hm}{kn}}|^j}.
\end{equation*}
Say $hm \geq kn+1$.  Then
\begin{equation*}
\Big|\log\frac{hm}{kn}\Big| \geq \log\Big(1 + \frac{1}{kn}\Big) \geq \frac{1}{2kn} \geq \frac{1}{2 \sqrt{hkmn}}.
\end{equation*}
The same inequality holds in case $kn \geq hm+1$, by symmetry.  The error terms from $hm \neq kn$ are then easily bounded by $O(T^{-A})$ for arbitrarily large $A$.
\end{proof}

\section{Proof of Lemma \ref{lemma:Ialphabeta}}
Inserting the definition of the mollifier $\psi$, we have
\begin{equation*}
I(\alpha, \beta) = \sum_{h, k \leq M} \frac{\mu(h) \mu(k) }{\sqrt{hk}} P\Big(\tfrac{\log{M/h}}{\log{M}} \Big) P\Big(\tfrac{\log{M/k}}{\log{M}} \Big) \intR w(t) \big(\tfrac{h}{k}\big)^{-it}  \zeta(\thalf + \alpha + it) \zeta(\thalf + \beta - it) dt.
\end{equation*}
According to Lemma \ref{lemma:twisted}, write $I(\alpha, \beta) = I_1(\alpha, \beta) + I_2(\alpha, \beta) + O(T^{-A})$.  Explicitly,
\begin{equation}
\label{eq:I1formula}
 I_1(\alpha, \beta) = \sum_{h, k \leq M} \frac{\mu(h) \mu(k) }{\sqrt{hk}} P\Big(\tfrac{\log{M/h}}{\log{M}} \Big) P\Big(\tfrac{\log{M/k}}{\log{M}} \Big) \sum_{hm = k n} \frac{1}{m^{\half + \alpha} n^{\half + \beta}} \intR V_{\alpha, \beta}(mn, t) w(t) dt.
\end{equation}
Notice that $I_2(\alpha, \beta)$ is obtained by replacing $\alpha$ with $-\beta$, $\beta$ with $-\alpha$, and multiplying by $X_{\alpha, \beta, t} = T^{-\alpha - \beta}(1 + O(L^{-1}))$.  That is, $I(\alpha, \beta) = I_1(\alpha, \beta) + T^{-\alpha-\beta} I_1(-\beta, -\alpha) + O(T/L)$.

\begin{mylemma}
\label{lemma:I1approx}
We have $I_1(\alpha, \beta) = c_1(\alpha, \beta) \widehat{w}(0) + O(T/L)$, uniformly on any fixed annuli such that $\alpha, \beta \asymp L^{-1}$, $|\alpha + \beta| \gg L^{-1}$, where
\begin{equation}
\label{eq:c1}
 c_1(\alpha, \beta) = \frac{1}{(\alpha + \beta) \log{M}} \frac{d^2}{dx dy}  M^{\alpha x + \beta y} \int_0^{1}  P(x+u) P(y+u) du \Big|_{x=y=0}.
\end{equation}
\end{mylemma}
\noindent {\bf Remark}.  Note that $c_1(\alpha,\beta)$ can be alternatively expressed as
\begin{equation}
\label{eq:c1alt}
c_1(\alpha, \beta) =  \frac{1}{(\alpha + \beta) \log{M}} \int_0^{1} (P'(u) +\alpha \log{M} \; P(u))(P'(u) + \beta \log{M} \; P(u)) du.
\end{equation}
We prove Lemma \ref{lemma:I1approx} in Section \ref{section:I1calc}. 
\begin{proof}[Proof that Lemma \ref{lemma:I1approx} implies Lemma \ref{lemma:Ialphabeta}]
By adding and subtracting the same thing, we have
\begin{equation*}
 I(\alpha, \beta) = [I_1(\alpha, \beta) + I_1(-\beta, -\alpha)] + I_1(-\beta, -\alpha)(T^{-\alpha-\beta}-1) + O(T/L).
\end{equation*}
We treat the two terms above differently.

We first compute the term in brackets using \eqref{eq:c1alt}, getting
\begin{equation*}
 c_1(\alpha, \beta) + c_1(-\beta, -\alpha) =  \int_0^{1} 2 P'(u) P(u) du = 1.
\end{equation*}
As for the second term, we have from \eqref{eq:c1} that
\begin{equation*}
(T^{-\alpha-\beta} -1) c_1(-\beta, -\alpha) =  \frac{1 - T^{-\alpha-\beta}}{(\alpha + \beta) \log{M}}  \frac{d^2}{dx dy}  M^{-\beta x - \alpha y} \int_0^{1} P(x+u) P(y+u) du \Big|_{x=y=0}. 
\end{equation*}
Note that
\begin{equation*}
 \frac{1-T^{-\alpha-\beta}}{(\alpha+\beta)\log{M}} = \frac{1}{\theta} \int_0^{1} T^{-v(\alpha+\beta)} dv.
\end{equation*}

Gathering the formulas gives \eqref{eq:c} although with the additional restriction that $|\alpha + \beta| \gg L^{-1}$.  However, the holomorphy of $I(\alpha, \beta)$ and $c(\alpha, \beta)$ with $\alpha, \beta \ll L^{-1}$ implies that the error term is also holomorphic in this region.  The maximum modulus principle extends the error term to this enlarged domain.
\end{proof}

\section{Proof of Lemma \ref{lemma:I1approx}}
A Mellin formula gives for $1 \leq h \leq M$ and $i=1,2, \dots$
\begin{equation}
\label{eq:PMellin}
\leg{\log{M/h}}{\log{M}}^i = \frac{i!}{(\log{M})^i} \frac{1}{2\pi i} \int_{(1)} \leg{M}{h}^v \frac{dv}{v^{i+1}}.
\end{equation}
Using \eqref{eq:PMellin} and \eqref{eq:V} in \eqref{eq:I1formula}, we have
\begin{multline*}
I_1(\alpha, \beta) = \intR w(t) \sum_{i,j} \frac{a_i a_j i! j!}{(\log{M})^{i+j}} \sum_{hm=kn} \frac{\mu(h) \mu(k)}{h^{\half} k^{\half} m^{\half + \alpha} n^{\half + \beta}} 
\\
\leg{1}{2 \pi i}^3 
\int_{(1)} \int_{(1)} \int_{(1)} 
\frac{M^{u+v}}{h^v k^u} \frac{g_{\alpha,\beta}(s,t)}{(mn)^s}\frac{G(s)}{s} ds \frac{du \; dv}{u^{i+1} v^{j+1}}.
\end{multline*}

\label{section:I1calc}
We compute the sum over $h,k,m,n$ as follows
\begin{equation}
\label{eq:arithmeticalfactor}
\sum_{hm=kn} \frac{\mu(h) \mu(k)}{h^{\half+v} k^{\half+u} m^{\half + \alpha+s} n^{\half + \beta+s}} = \frac{\zeta(1+u+v) \zeta(1+\alpha + \beta + 2s)}{\zeta(1+\alpha + u + s) \zeta(1 + \beta + v + s)} A_{\alpha, \beta}(u,v,s), 
\end{equation}
where the arithmetical factor $A_{\alpha, \beta}(u,v,s)$ is given by an absolutely convergent Euler product in some product of half planes containing the origin.  Next we move the contours to $\text{Re}(u) = \text{Re}(v) = \delta$, and then $\text{Re}(s) = -\delta + \varepsilon$ (for $\delta > 0$ sufficiently small so that the arithmetical factor is absolutely convergent), crossing a pole at $s=0$ only since $G(s)$ vanishes at the pole of $\zeta(1+\alpha + \beta + 2s)$.  Since $M \leq T^{\theta}$ with $\theta < \half$, and $t \geq T/2$, the new contour of integration gives $O(T^{1-\varepsilon})$ for sufficiently small $\varepsilon > 0$, using \eqref{eq:Stirling}.  Thus
\begin{equation}
\label{eq:I1approx}
I_1(\alpha, \beta) =  \widehat{w}(0) \zeta(1 + \alpha + \beta) \sum_{i, j} \frac{a_i a_j i! j!}{(\log{M})^{i+j}} J_{\alpha, \beta}(M)  + O(T^{1-\varepsilon}),
\end{equation}
where
\begin{equation*}
J_{\alpha, \beta}(M) = \leg{1}{2 \pi i}^2 \int_{(\varepsilon)} \int_{(\varepsilon)}M^{u+v} \frac{ \zeta(1+u+v) A_{\alpha, \beta}(u,v,0) }{\zeta(1+\alpha + u) \zeta(1 + \beta + v)}   \frac{du \; dv}{u^{i+1} v^{j+1}}.
\end{equation*}

\begin{mylemma}
\label{lemma:Jasymp}
We have, uniformly for $\alpha, \beta \ll L^{-1}$,
\begin{equation}
\label{eq:Jasymp}
J_{\alpha, \beta}(M) = 
 \frac{(\log{M})^{i+j-1}}{i! j!} \frac{d^2}{dx dy} 
M^{\alpha x + \beta y} \int_{0}^{1} (x+u)^i (y + u)^j du \Big|_{x=y=0}
+ O(L^{i+j-2}).
\end{equation}
\end{mylemma}
\noindent Lemma \ref{lemma:I1approx} follows directly from Lemma \ref{lemma:Jasymp} by summing over $i$ and $j$, and taking a Taylor expansion of $\zeta(1+\alpha+\beta)$.
\begin{proof}[Proof of Lemma \ref{lemma:Jasymp}]
We begin by using the Dirichlet series for $\zeta(1 + u + v)$ and reversing the order of summation and integration to get
\begin{equation*}
J_{\alpha, \beta}(M) = \sum_{n \leq M} \frac{1}{n} \leg{1}{2 \pi i}^2 \int_{(\varepsilon)} \int_{(\varepsilon)} \leg{M}{n}^{u+v} \frac{ A_{\alpha, \beta}(u,v,0) }{\zeta(1+\alpha + u) \zeta(1 + \beta + v)}   \frac{du \; dv}{u^{i+1} v^{j+1}}.
\end{equation*}
Using the standard zero-free region of $\zeta$ and upper bound on $1/\zeta$ (see \cite{T}, Theorem 3.8 and (3.11.8)), we obtain that $J_{\alpha, \beta}(M)$ equals the residue at $u = v =0$ plus an error of size
\begin{equation*}
\sum_{n \leq M} \frac{1}{n} (1 + \log\tfrac{M}{n})^{-2} \ll 1 \ll L^{i+j-2}.
\end{equation*}
For computing the residue we take contour integrals of radius $\asymp L^{-1}$ and use the Taylor approximation 
\begin{equation*}
\frac{A_{\alpha, \beta}(u,v,0)}{\zeta(1+\alpha+u) \zeta(1+\beta + v)} = (\alpha + u)(\beta + v) A_{0,0}(0,0,0) + O(L^{-3}).
\end{equation*}
We show in Section \ref{section:arithmetical} below that $A_{0,0}(0,0,0) = 1$, a result we now use freely.
Thus
\begin{equation*}
J_{\alpha, \beta}(M) = \sum_{n \leq M} \frac{1}{n} \leg{1}{2\pi i}^2 \left(\oint \leg{M}{n}^u (\alpha + u) \frac{du}{u^{i+1}} \right) \left(\oint \leg{M}{n}^v (\beta + v) \frac{dv}{v^{j+1}} \right) + O(L^{i+j-2}),
\end{equation*}
where the contours are circles of radius $1$ around the origin.

We compute these two integrals exactly.  Suppose $a > 0$.  Then
\begin{equation*}
\frac{1}{2 \pi i} \oint a^{u} (\alpha + u) \frac{du}{u^{l+1}} = \frac{d}{dx} \Big[e^{\alpha x} \frac{1}{2 \pi i} \oint (a e^x)^u \frac{du}{u^{l+1}} \Big]_{x=0} = \frac{1}{l!} \frac{d}{dx} e^{\alpha x} (x + \log{a})^l \Big|_{x=0}.
\end{equation*}
Thus
\begin{equation*}
 J_{\alpha, \beta}(M) = \frac{1}{i! j!} \frac{d^2}{dx dy} 
e^{\alpha x + \beta y} \sum_{n \leq M} \frac{1}{n} (x + \log(M/n))^i (y + \log(M/n))^j 
\Big|_{x=y=0}
+ O(L^{i+j-2}).
\end{equation*}
Note that
\begin{equation*}
 \frac{d}{dx}  e^{\alpha x} (x + \log(M/n))^i \Big|_{x=0} = 
\frac{(\log{M})^i}{\log{M}} \frac{d}{dx}M^{\alpha x} \Big(x + \frac{\log(M/n)}{\log{M}} \Big)^i \Big|_{x=0},
\end{equation*}
so that by summing over $i$ and $j$ we have
\begin{multline*}
 J_{\alpha, \beta}(M) = \frac{(\log{M})^{i+j-2}}{i! j!} \frac{d^2}{dx dy} 
M^{\alpha x + \beta y} \sum_{n \leq M} \frac{1}{n} \Big(x + \frac{\log(M/n)}{\log{M}} \Big)^i \Big(y + \frac{\log(M/n)}{\log{M}} \Big)^j \Big|_{x=y=0}
\\
+ O(L^{i+j-2}).
\end{multline*}
By the Euler-Maclaurin formula, we can replace the sum over $n$ by a corresponding integral without introducing a new error term (this requires some thought).  That is,
\begin{multline*}
 J_{\alpha, \beta}(M) = \frac{(\log{M})^{i+j-2}}{i! j!} \frac{d^2}{dx dy} 
M^{\alpha x + \beta y} \int_{1}^{M} r^{-1} \Big(x + \frac{\log(M/r)}{\log{M}} \Big)^i \Big(y + \frac{\log(M/r)}{\log{M}} \Big)^j \Big|_{x=y=0}
\\
+ O(L^{i+j-2}).
\end{multline*}
Changing variables $r = M^{1-u}$ and simplifying finishes the proof.
\end{proof}

\section{The arithmetical factor}
\label{section:arithmetical}
Here we verify that $A_{0,0}(0,0,0) = 1$ as claimed in the proof of Lemma \ref{lemma:Jasymp}.  The proof is surprisingly easy.  We show that $A_{0,0}(s,s,s) =1$ for all $\text{Re}(s) > 0$.  From \eqref{eq:arithmeticalfactor} we have
\begin{equation*}
 A_{0,0}(s,s,s) = \sum_{hm=kn} \frac{\mu(h) \mu(k)}{(hkmn)^{\half + s}},
\end{equation*}
noting that the ratios of zeta's on the right hand side of \eqref{eq:arithmeticalfactor} cancel.  The result now follows instantly from the M\"{o}bius formula.

\end{document}